\documentclass[12pt]{article}
\usepackage[a4paper, left=2cm, right=2cm, bottom=3cm, top=3cm]{geometry}
\usepackage{amsmath}
\usepackage{srcltx}
\usepackage{amsfonts}
\usepackage{amssymb}
\usepackage{amsthm}
\usepackage{lipsum}
\usepackage{psfrag}
\usepackage{verbatim}
\usepackage{graphicx} 
\usepackage{color}
\usepackage{faktor}
\providecommand{\U}[1]{\protect\rule{.1in}{.1in}}

\newtheorem{theorem}{Theorem}[section]

\newtheorem{corollary}[theorem]{Corollary}

\newtheorem{definition}[theorem]{Definition}
\newtheorem{example}[theorem]{Example}

\newtheorem{lemma}[theorem]{Lemma}

\newtheorem{proposition}[theorem]{Proposition}
\newtheorem{remark}[theorem]{Remark}

\newcommand{\adj}{\mathop{\rm ad}\nolimits}

\newcommand{\bpsi}{\boldsymbol{\psi}}

\newcommand{\R}{\ensuremath{\mathbb{R}}}

\newcommand{\N}{\ensuremath{\mathbb{N}}}

\newcommand{\Z}{\ensuremath{\mathbb{Z}}}

\begin{document}
\title{An accessibility condition for discrete-time linear systems on Lie groups}
\author{Thiago Matheus Cavalheiro\\Departamento de Matem\'{a}tica, Universidade Estadual de Maring\'{a}\\Maring\'{a}, Brazil.
\and Alexandre J. Santana\\Departamento de Matem\'{a}tica, Universidade Estadual de Maring\'{a}\\ Maring\'a, Brazil
\and Eduardo Celso Viscovini\\Departamento de Matem\'{a}tica, Universidade Estadual de Maring\'{a}\\Maring\'{a}, Brazil }
\maketitle

\begin{abstract} 

In this paper, we characterize the accessibility of discrete-time linear control systems on
Lie groups. Using an exceptional notion of derivative, we construct a subalgebra $\mathfrak{h}$ based
on the infinitesimal automorphism of the system such that if its dimension is maximal,
the system is accessible. Our criteria provide simple conditions  in a general context for the discrete-time case. Additionally, we prove a sufficient
condition for local controllability at the identity using the infinitesimal automorphism,
akin to the ad-rank condition in the continuous case.
\end{abstract}
%%%%%%%%%%%%%%%%%%%%%%%%%%%%%%%%%%%%%%%%%%%%%%%%%%%%%%%%%%%%%%%%%%%%%%%%%%%%%%%%%%%%%%%

%%%%%%%%%%%%%%%%%%%%%%%%%%%%%%%%%%%%%%%%%%%%%%%%%%%%%%%%%%%%%%%%%%%%%%%%%%%%%%%%%%%%%%%

%%%%%%%%%%%%%%%%%%%%%%%%%%%%%%%%%%%%%%%%%%%%%%%%%%%%%%%%%%%%%%%%%%%%%%%%%%%%%%%%%%%%%%%

\section{Introduction} 

A control system, in general, is a set of mathematical concepts that describes the behavior of a dynamic system  on a state space $M$ under the influence of control inputs. One of the most fundamental concepts in the control systems theory is controllability, which means that the entire $M$ can be reached by trajectories (in positive time) of the control system from any starting point. 	For the special case of linear control systems on $\mathbb{R}^{n}$ there exists a simple algebraic tool for deciding controllability, the presently known algebraic necessary and sufficient conditions for nonlinear	systems are still far from complete (see e.g. Elliot \cite{elliot}. Moreover, it is well known that whatever
	necessary and sufficient conditions eventually are
	found for non-linear cases, these are likely to be rather hard to check (see Kawski \cite{kawski} and Sontag \cite{sontag harder}). However, note that there are several important partial results for systems with
	some specific classes of state spaces, for systems restrict to some regions of the state space, as well as results about weaker controllability concepts (see Ayala and Da Silva \cite{AyalaandDaSilva2}, Colonius and Kliemann \cite{FritzKliemann}, Sontag \cite{Son98} and references therein). Two examples of these weaker controllability concepts are the local controllability, that is, the system can be steered from an initial state to any neighboring point and back, and 
	accessibility condition, which means that the system can be steered from an initial state to some full-dimensional final set our equivalently, the set of reachable points (from any initial state) has non-empty interior (see Dıaz-Seoane et al \cite{biosystems} and Sontag \cite{Son98}). Recall the following fundamental relation between accessibility and controllability: if
	the system is controllable, the trajectories are all open, implying accessibility;  On the other
	hand, if the system is only accessible, the controllability may or may not happen. 
	
	 In case of continuous-time linear control systems on Lie group several works in the 2000s began to be published studying conditions for accessibility and controllability  (see e.g.  Ayala and Tirao \cite{AyalaAndTirao}, Do Rocio et all \cite{DoRocioSantanaAndVerdi}, Ayala and Da Silva \cite{ayalaeadriano} and 	\cite{AyalaandDaSilva2}, Da Silva
	\cite{adriano}, Jouan
	\cite{jouan} and Ayala and San Martin
	\cite{SanMartinandAyala}). 
 
 A classical concept related with controllability is the Lie algebra rank condition - LARC: consider the continuous-time linear systems on a Lie group $G$
\begin{equation}\label{linearcontinuoustimeonG}
    \dot{x}(t) = \mathcal{X}(x(t)) + \sum_{j=1}^{m} X_j(x(t))u_j(t),
\end{equation}
where $u_j(t)$ are piecewise constant functions with image on some compact neighborhood of $0 \in \mathbb{R}^m$, $\mathcal{X}$ is a vector field on $G$ contained in the normalizer of its Lie algebra $\mathfrak{g}$ and $X_j$ are right-invariant vector fields. Denote the solution of the system  starting at $x \in G$ by $\varphi(t,x,u)$.
The reachable set from  identity $e \in G$ at time $t > 0$ and the reachable set from  $e$ are given, respectively, by
$
\mathcal{R} _t(e)$ $ :=$ $ \{\varphi(t,e,u) \mid u=(u_{1} , \ldots , u_{m}) \}$ and $\mathcal{R}$ $  :=$ $ \bigcup_{t>0} \mathcal{R} _t(e).
$

Consider the derivation $\mathcal{D}(X) = [\mathcal{X},X]$ of $\mathfrak{g}$ and denote $\mathfrak{h}$ as the Lie subalgebra of $\mathfrak{g}$ generated by $\{X_j : j = 1, \ldots, m\}$. According \cite{adriano1}, We say that the system satisfies LARC if  $\mathfrak{g}$ is the smallest $\mathcal{D}$-invariant subalgebra containing $\mathfrak{h}$, which implies  the accessibility condition of system (\ref{linearcontinuoustimeonG}).
There are several results which can be obtained from this context. For example, San Martin and Ayala \cite{SanMartinandAyala} proved that for compact connected Lie groups, the system is controllable if, and only if, the LARC condition is satisfied. This also holds for abelian Lie groups \cite[Corollary 3.6]{AyalaAndTirao}.  For three-dimensional solvable nonnilpotent Lie groups, Ayala and Da Silva \cite{AyalaandDaSilva2} proved that the LARC condition and the study of the eigenvalues of the derivation $\mathcal{D}$ are essential for controllability on those systems as well.

In the matter of local controllability at $e \in G$  (which is equivalent to say $e \in \hbox{int}\mathcal{R}$), we highlight two well-known results that give sufficient conditions  (ad-rank condition) for local controllability in $e$. The unconstrained case were studied by Jouan \cite{jouan} and we will not discuss it here. First, recall that system (\ref{linearcontinuoustimeonG}) is said to satisfy the ad-rank condition if the vector subspace $V = span\{\mathcal{D}^j(X_i(e)): i = 1,\ldots,m, j \in \N_0\}$ coincides with $\mathfrak{g}$ (see \cite{adriano}). Nevertheless, \cite[Theorem 3.5]{AyalaAndTirao} ensures that if the system (\ref{linearcontinuoustimeonG}) satisfies the ad-rank condition, then it is locally controllable. There are several others interesting works relating ad-rank condition with controllability. We recall that Do Rocio, Santana, and Verdi \cite{DoRocioSantanaAndVerdi} compare the ad-rank condition on the constrained case with some other classical results about local controllability.  We cite also Da Silva \cite{adriano} and Ayala and Da Silva \cite{ayalaeadriano} that proved that the ad-rank condition and the study of the eigenvalues of the derivation $\mathcal{D}$ are essential for global controllability. 
 
More recently, Colonius, Cossich and Santana \cite{CCS1} introduced a  discrete-time version of the above system. And Cavalheiro, Cossich and Santana studied controllability for these discrete systems in \cite{TAJ1} and \cite{TAJ}.

The discrete-time linear control system presented in \cite{CCS1} is a special case of the following family of dynamic systems  on a manifold $M$:
	\begin{equation} \label{generaldiscrete}
		x_{k+1}=f(u_k,x_k), k \in {\mathbb N}_{0}={\mathbb N} \cup \{0\}, 
	\end{equation}
	for which $f: U \times M \longrightarrow M$ is a function defined on a non-empty compact neighborhood $U$ of $0$ in $\R^m$. Now consider the particular case of system (\ref{generaldiscrete}), the linear control system on $\mathbb{R}^{n}$,
	\begin{equation}
		\label{discreteLinearSystemInRn}
		x_{k+1}=Ax_k+Bu_k,
	\end{equation}
	where $A\in{\rm Gl}(n,\mathbb{R})$, $B\in \mathbb{R}^{n\times m}$ and $U\subset\mathbb{R}^m$ is a compact convex neighborhood of the origin. Then 	the condition for controllability is well-known and the main result proved by Kalman, Ho, and Narendra \cite{KalmanHoNarendra} states that the controllability of the system (\ref{discreteLinearSystemInRn}) is equivalent to the image of the matrix  $\mathcal{K} = (A^{n-1}B \hbox{ }... \hbox{ }AB \hbox{ }B)$ has dimension $n$. This condition is called \textit{Kalman condition} and is also valid for the continuous-time version on $\mathbb{R}^{n}$.

	In this present paper, we work with   a special case of the system (\ref{generaldiscrete}), the  discrete-time linear  systems on a connected Lie group $G$, defined  in \cite{CCS1} as:
	\begin{equation}
		\label{discreteLinearSystem}
		 g_{k+1} = f(u_k,g_k), k \in \N_0,
	\end{equation}
	where $f: U \times G \longrightarrow G$ must satisfy the following properties 
	\begin{itemize}
		\item[1-] $f(u,g) = f(u,e) f(0,g)$, for every $g \in G$ and $u \in U$, 
		\item[2-] The function $f_0: = f(0,\cdot): G \longrightarrow G$ is an automorphism.  
	\end{itemize}
	
Note that this system is a discrete-time version of the continuous-time linear control systems on a Lie group $G$ (see e.g. \cite{AyalaandDaSilva2} and references therein) and generalizes the  system (\ref{discreteLinearSystemInRn}). In the continuous-time case, accessibility suffices to ensure the existence of subsets with non-empty interior and approximate controllability, namely control sets (Da Silva and Rojas  \cite{adriano1}).  For the system (\ref{discreteLinearSystem}), more conditions are needed to guarantee their existence (see Cavalheiro, Cossich and Santana \cite{TAJ}).

A result characterizing the accessibility of discrete control systems in a more general setting was proven by Jakubcyk and Sontag \cite{JakubczykAndSontag}, in terms of the partial derivatives of the function $f$. In this paper, we present two main results. One of them characterizes the accessibility of the system in terms of either the smallest subgroup containing the set $f(U,e)$ or the derivatives of the function $u \mapsto f(u,e)$ (see Theorem 3.9). The second result concerns local controllability at $e$ and consists of a result similar to the ad-rank condition in the constrained case, although for discrete-time linear systems. To state and prove both results, we define a new notion of derivative for functions between Lie groups, using right-invariant vector fields as a reference. This derivative is well-behaved with respect to the functions that are usually considered in the context of Lie groups, such as left and right translations and group homomorphisms (see Examples 3.6 and 3.7), and it greatly simplifies the calculations. The advantage of our results is that we are able to characterize accessibility using simpler conditions more in line with Lie theory.

%It also generalizes the notion of derivative used in real analysis, where the derivative of a function f:Rk→Rlf:\mathbb{R}^k\rightarrow \mathbb{R}^l is seen as a function from Rk\mathbb{R}^k into L(Rk,Rl)\mathcal{L}(\mathbb{R}^k,\mathbb{R}^l) (see Example ???????????????????????????????????????????????????????????????????????????\ref{exempleGeneralizesRn}).
	
This paper is divided as follows: in the section \ref{prelimi} we expose the first concepts about control systems, constructing the main definitions and concepts as well. In the section \ref{access} we show the main results, initially showing some basic properties about linear systems on Lie groups, followed by the definition of the derivative $\hat d$ and some of its properties, and then the theorems about accessibility and the ad-rank condition, finalizing with some examples on the Heisenberg group, the special matrix Lie subgroup $\hbox{SL}(2,\R)$ and the affine two-dimensional Lie group $\hbox{Aff}_2(\R)$.       

\section{Preliminaries}\label{prelimi}

In this section, we define some properties of control systems and establish standartized terminology to be employed throughout this paper. In the sequence, in the discrete-time control system (\ref{generaldiscrete}) we consider    $M$ a smooth $n-$dimensional Riemannian manifold equipped with a canonical metric $d$ and the control range satisfying $U \subset \overline{\hbox{int}U}$.   We assume that for an open set $\hat{U}$ containing $U$, the function $f: \hat{U} \times M \longrightarrow M$ is a $\mathcal{C}^{\infty}$ map as well. Also, we will consider that $f_u: M \longrightarrow M$ is a invertible map, that is, the function $f^{-1}_u: M \longrightarrow M$ is well-defined.  

Given an initial condition $x \in M$, considering the set $\mathcal{U} = \prod_{i \in \Z} U$, the solution $\varphi: \Z \times M \times \mathcal{U} \longrightarrow M$ of the system (\ref{generaldiscrete}) is given by
\begin{equation*}
    \varphi(k,x_0,u) = 
    \left\{
    \begin{array}{cc}   
        f_{u_{k-1}} \circ ... \circ f_{u_0}(x_0),& k > 0\\
        x_0,& k=0\\
        f_{u_{k}}^{-1} \circ ... \circ f_{u_{-1}}^{-1}(x_0),& k < 0
    \end{array}
    \right.
\end{equation*}

Also, considering the map $\Theta: \Z \times \mathcal{U} \longrightarrow \mathcal{U}$, defined by $\Theta_k((u_i)_{i \in \Z}) = (u_{i+k})_{i \in \Z}$, the solution $\varphi$ satisfies the cocycles property, which means that
$\varphi(k+t, x,u) =$ $\varphi(k,\varphi(t,x,u), \Theta_t(u))=$ $\varphi(t,\varphi(k,x,u), \Theta_k(u)), \forall k,t \in \Z$.
The solution $\varphi$ also satisfy the following property: if $ts > 0$ in $\Z$, given $u,v \in \mathcal{U}$, there is a $w \in \mathcal{U}$ such that 
$    \varphi(t,\varphi(s,g,u),v) = $ $\varphi(t + s, g, w), \forall g \in M$. 
The shift space $\mathcal{U}$ is compact using the canonical topology.

Now we present a sequence of definitions and concepts necessary in this paper.

\begin{definition}For $x \in M$, the set of points reachable and controllable from $x$ up to time $k > 0$ in $\mathbb{N}$ are
\begin{eqnarray*}
    \mathcal{R}_k(x) = \{y \in M: \hbox{ there is }u \in \mathcal{U} \hbox{ with }\varphi(k,x,u)=  y\}\\
    \mathcal{C}_k(x) = \{y \in M: \hbox{ there is }u \in \mathcal{U} \hbox{ with }\varphi(k,y,u)= x\}\\
\end{eqnarray*}
The sets $\mathcal{R}(x) = \bigcup_{k \in \mathbb{N}} \mathcal{R}_k(x)$ and $\mathcal{C}(x) = \bigcup_{k \in \mathbb{N}} \mathcal{C}_k(x)$ denote the reachable set and the controllable set from $x$ respectively.  
\end{definition}

\begin{definition}\label{regular}For each $k \in \N$, consider the function $G_k(x,u) = \varphi(k,x,u)$.  A pair $(x,u) \in M \times \hbox{int}U^k$ is called regular if $\hbox{rank}\left[\frac{\partial}{\partial u} G_k(x,u)\right] = \hbox{dim}M.$ We denote by 
\begin{equation*}
    \hat{\mathcal{R}}_k(x) = \{\varphi(k,x,u): (x,u) \in M\times \hbox{int}U^k \hbox{ is regular}\}.
\end{equation*}
the regular reachable set of $x \in G$ up to time $k \in \N$ and $\hat{\mathcal{R}}(x) = \bigcup_{k \in \N} \hat{\mathcal{R}}_k(x)$ the regular reachable set of $x \in M$. 
\end{definition}

In particular, the set $\hat{\mathcal{R}}(x)$ is open, for every $x \in M$. From this we have the following definition.
\begin{definition}
The system (\ref{generaldiscrete}) is forward accessible (resp. backward accessible) if $\hbox{int}\mathcal{R}(x) \neq \emptyset$ (resp. $\hbox{int}\mathcal{C}(x) \neq \emptyset$), for all $x \in M$. And it is said to be accessible if both conditions are satisfied.
\end{definition}

Another important definition in this work is the ad-rank condition for the system (\ref{generaldiscrete}), but as it is necessary some construction to introduce it  we present this definition later (see the following Definition \ref{adrankdef}).

Controllability is the most desirable property for control systems.
\begin{definition}
 We say that the system (\ref{generaldiscrete}) is controllable if   for any $x,y \in M$, there are $k \in \N$ and $u \in \mathcal{U}$ such that $\varphi(k,x,u) = y$.
\end{definition}
 When the controllability is not achievable, we can search for the maximal regions where controllability-like properties may occur, as in the next definition. 
\begin{definition}Consider the control system (\ref{generaldiscrete}). A control set $D$ is a set satisfying the following properties 
\begin{itemize}
    \item[1-] For any $x \in D$, there is a $u \in \mathcal{U}$ such that $\varphi(k,x,u) \in D$, for any $k \in \N$. 
    \item[2-] $D \subset \overline{\mathcal{R}(x)}$, for any $x \in D$. 
    \item[3-] $D$ is maximal with such properties.  
\end{itemize}
\end{definition}

Control sets has a very important role in control systems (see e.g. \cite{FritzKliemann}) and specially in linear control systems in continuous-time (see e.g. \cite{AyalaandDaSilva3}). 

Now consider the discrete-time linear system (\ref{discreteLinearSystem}).
Note that its map $f$ can be defined using the translations of $G$. In fact, given $u \in U$, as $f_u(e) \in G$, we can write $f_u(g)$ as 
\begin{equation}\label{transl1}
    f_u(g) = f_u(e)f_0(g) = L_{f_u(e)}(f_0(g)), 
\end{equation}
where $L_{f_u(e)}$ is the left translation by the element $f_u(e)$. Considering the expression above, the inverse of $f_u$ is given by 
\begin{equation}
    (f_u)^{-1}(g) = f_0^{-1}\circ L_{(f_u(e))^{-1}}(g)= f_0^{-1}((f_u(e))^{-1} \cdot g). 
\end{equation}

Then, we can conclude that $f_u$ is a diffeomorphism of $G$, for any $u \in U$. The solutions in any $g \in G$ can also be defined in terms of translations of the solution at the neutral element by the automorphism $f_0$, see the next proposition proved in (Colonius, Cossich and Santana \cite{CCS1}). 

\begin{proposition}\label{prop52} Consider the above discrete-time linear control system defined on a Lie group $G$. Then for all $g \in G$ and $u = (u_i)_{i \in \Z} \in \mathcal{U}$
\begin{equation*}
    \varphi(k,g,u) = \varphi(k,e,u)f_0^k(g). 
\end{equation*}
\end{proposition}

Take the reversed-time system 
\begin{equation}\label{reversed-time-system}
     g_{k+1} = \hat{f}_{u_k}(g_k), k \in \N_0,
\end{equation}
given by the map $\hat{f}_u(g) = f_u^{-1}(e)f_0^{-1}(g)$, denote by $\mathcal{R}^*_k$ and $\mathcal{C}^*_k$ its reachable and controllable sets up to time $k$ of $e$ of the system (\ref{reversed-time-system}), its is proved in \cite{TAJ} the following result. 

\begin{lemma}\label{setsinvsys}It holds that $\mathcal{R}_k^*=\mathcal{C}_k$ and $\mathcal{R}_k=\mathcal{C}_k^*$, for all $k\in\mathbb{N}$.
\end{lemma}

\section{Conditions for Accessibility}\label{access}

Let us consider a connected Lie group $G$ with Lie algebra $\mathfrak{g}$ and the discrete-time linear system (\ref{discreteLinearSystem})
 where $U$ is a compact convex neighborhood of $0 \in \R^m$. As said before, the function $f_u: G \longrightarrow G$ is a diffeomorphism for any $u \in U$ and $f_0$ is a automorphism of $G$. Then the system is defined for any $k \in \Z$. Denote by $\mathfrak{g}$ the Lie algebra of $G$, endowed with the Lie bracket 
\begin{equation}\label{liebracket}
    [X,Y] = \frac{\partial^2}{\partial t \partial s} \left(X_{-t} \circ Y_s \circ X_t\right)\bigg|_{t=s=0}. 
\end{equation}
where $X_t$ and $Y_t$ are the respectives solutions of $X$ and $Y$ at the time $t \in \R$. It is well-known (see \cite{sanmartin1}) that $[X,Y] = 0$ if, and only if, $\exp{tX}\exp{sY} = \exp{sY}\exp{tX}$ for all $t,s\in\mathbb{R}$. 

Considering the reachable sets $\mathcal{R}_k(e)$, $\mathcal{R}_{\leq k}(e) = \{\varphi(t,e,u): t \in [0,k] \cap \N, u \in \mathcal{U}\}$ and $\mathcal{R}(e) = \bigcup_{k \in \N}\mathcal{R}_k(e)$, it is easy to see that $e \in \mathcal{R}_k(e)$ for any $k \in \N.$ Besides, using the notation $\mathcal{R}(e) = \mathcal{R}$, $\mathcal{R}_k(e) = \mathcal{R}_k$ and $\mathcal{R}_{\leq k} = \mathcal{R}_{\leq k}(e)$, we get the following property whose proof can be found in \cite{TAJ}.

\begin{proposition}\label{reachablesetprop}The reachable set $\mathcal{R}$ satisfy the following properties: 
\begin{itemize}
    \item[1-] Given $\tau \geq 1$ in $\N$, then $\mathcal{R}_{\tau} = \mathcal{R}_{\leq \tau}$. 
    \item[2-] Given $0 < \tau_1 \leq \tau_2$ in $\N$, then $\mathcal{R}_{\tau_1} \subset \mathcal{R}_{\tau_2}$. 
    \item[3-] If $g \in G$, then $\mathcal{R}_{\tau}(g) = \mathcal{R}_{\tau} f_0^{\tau}(g)$. 
    \item[4-] If $\tau_1, \tau_2 \in \N$, then $\mathcal{R}_{\tau_1 + \tau_2} = \mathcal{R}_{\tau_1} f_0^{\tau_1}(\mathcal{R}_{\tau_2}) = \mathcal{R}_{\tau_2} f_0^{\tau_2}(\mathcal{R}_{\tau_1})$. 
    \item[5-] For any $u \in \mathcal{U}$, $g \in G$ and $k \in \N$, then 
    $\varphi(k,\mathcal{R}(g),u) \subset \mathcal{R}(g)$
    \item[6-] $e \in \hbox{int}\mathcal{R}$ if and only if $\mathcal{R}$ is open. 
\end{itemize}
\end{proposition}

Regarding the set $\mathcal{R}$ and also denoting $\mathcal{C}_k(e) = \mathcal{C}_k$, we have the following connection between the sets $\mathcal{R}_k$ and $\mathcal{C}_k$. 

\begin{proposition}\label{openess}Consider the linear system (\ref{discreteLinearSystem}). Then $\hbox{int}\mathcal{R}_k \neq \emptyset$ if, and only if, $\hbox{int}\mathcal{C}_k \neq \emptyset.$  
\end{proposition}

\begin{proof}
 Let us suppose $\hbox{int}\mathcal{R}_k \neq \emptyset$ and consider the automorphism $f_0$ of the system (\ref{discreteLinearSystem}). Take $g \in \hbox{int}\mathcal{R}_k$. Hence, there is a $u \in \mathcal{U}$ such that 
$ g = \varphi(k,e,u)$,
that is $\varphi(k,e,u)g^{-1} = e$. Using the properties of $\varphi$ we get 
$    \varphi(k,e,u)g^{-1} = $ 
 $\varphi(k,e,u)f_0^k(f_0^{-k}(g^{-1})) = $ $\varphi(k,f_0^{-k}(g^{-1}),u) = e$,
that is $f_0^{-k}(g^{-1}) \in \mathcal{C}_k$. Now, consider $V$ a neighborhood of $g$ such that $g \in V \subset \mathcal{R}_k.$ By the arguments above, $f^{-k}_0(V^{-1}) \subset \mathcal{C}_k$. The function $f_0^{-k}$ is an automorphism of $G$ and $V^{-1}$ is a neighborhood of $g^{-1}$. Then $f_0^{-k}(g^{-1}) \in \hbox{int}\mathcal{C}_k$.

Now, if $\hbox{int}\mathcal{C}_k\neq \emptyset$, let us take $g \in \hbox{int}\mathcal{C}_k$. Then, there are $k \in \N$ and $u \in \mathcal{U}$ such that $\varphi(k,g,u) = e$. Hence
$    \varphi(k,g,u) = \varphi(k,e,u)f_0^k(g) = e$,
and consequently $\varphi(k,e,u) = f_0^{k}(g^{-1})$. If $g \in V \subset \mathcal{C}_k$, for some neighborhood $V$ of $g$, by the previous argument, $f_0^k(V^{-1}) \subset \mathcal{R}_k$. Therefore $\hbox{int}\mathcal{R}_k \neq \emptyset.$
\end{proof}

\begin{remark}\label{openess1}For the case when $\hbox{int}\mathcal{R} \neq \emptyset$, in \cite{CCS2} it is cited that there is a $k_0 \geq 1$ such that $\hbox{int}\mathcal{R}_k \neq \emptyset$ for every $k \geq k_0$. Thus, the proposition above ensures that $\hbox{int}\mathcal{R} \neq \emptyset$ if, and only if, $\hbox{int}\mathcal{C} \neq \emptyset$. Also, if $e \in \hbox{int}\mathcal{R}$, by the proof above, $e \in \hbox{int}\mathcal{C}$. Therefore, $\mathcal{R}$ is open if, and only if, $\mathcal{C}$ is open 
\end{remark}

\begin{lemma}\label{lemmaaccessibility} The system (\ref{discreteLinearSystem}) is accessible if and only if $\hbox{int}\mathcal{R} \neq \emptyset$. 
\end{lemma}
\begin{proof}
    Under the hypothesis that $\hbox{int}\mathcal{R} \neq \emptyset,$ there is a $k_0 \in \N$ such that if $k \geq k_0$, $\hbox{int}\mathcal{R}_k \neq \emptyset.$ (see \cite[Page 3]{CCS2}). Consider $V$ an open set such that $V \subset \mathcal{R}_k$. For any $g \in G$, the set $Vf_0^k(g)$ is an open set and 
\begin{equation}
    V f_0^k(g) \subset \mathcal{R}_k f_0^k(g) = \mathcal{R}_k(g). 
\end{equation}

Then $\hbox{int}\mathcal{R}_k(g) \neq \emptyset$, for every $k \geq k_0$. This proves that the system (\ref{discreteLinearSystem}) is forward accessible if $\hbox{int}\mathcal{R} \neq \emptyset.$ For backward accessibility, using the fact of $\mathcal{C}_k(g) = \mathcal{C}_kf_0^{-k}(g)$, Proposition (\ref{openess}) ensures that (\ref{discreteLinearSystem}) is accessible (forward and backward) if $\hbox{int}\mathcal{R} \neq \emptyset.$ The other implications is true by definition.
\end{proof}

\subsection{Main results}

We introduce some of needed concepts and algebric generalities. Whenever we say differentiable we mean continuously differentiable. If $M$ is a manifold and $x\in M$, recall that any vector in the tangent space $T_xM$ of $M$ can be written as
$c'(0)=\left.\frac{d}{dt}\right|_{t=0}c(t)$
for some differentiable curve $c:(-\varepsilon,\varepsilon)\rightarrow M$ such that $c(0)=x$.

Let $G$ be a Lie group and $\mathfrak{g}$ its Lie algebra. It is well known that for $X\in\mathfrak{g}$ and $g\in G$ we have $X(g)$ $=$ $X_R(g)$ $=$ $d(R_g)_e(X)$ $=$ $\left.\frac{d}{dt}\right|_{t=0}e^{tX}g,$
where $X_R$ denotes the only right invariant vector field in $G$ satisfying $X_R(e)=X$. Furthermore, for $g\in G$ the map
$         \phi_g:\mathfrak{g} \rightarrow T_gG  \, , \,\,\, X\mapsto X(g) \, ,
   $
is a linear isomorphism between $\mathfrak{g}$ and $T_gG$.

Take $H$ a Lie group, $\mathfrak{h}$ its Lie algebra, and $f:O\subset H\rightarrow G$ a differentiable map with $O$ an open set in $H$. For any $h \in O$, we define the derivative $\hat d$ in $h$ by
$$
         \hat df_h:\mathfrak{h}\longrightarrow\mathfrak{g} \, , \,\,\,        X\mapsto \phi_{f(h)}^{-1}df_h(X(h)),
   $$
where $df_h:T_{h}H\rightarrow T_{f(h)}G$ denotes the usual derivative in the sense of manifolds. Thus, $\hat df_h(X)$ is defined as the vector field $Y\in \mathfrak{g}$ which satisfies $Y(f(h))=df_h(X(h))$. Alternatively, $\hat df_h(X)$ can be calculated from the equality

\begin{equation*}
    \hat df_h(X) = d(R^G_{f(h)^{-1}})_{f(h)} \circ df_h \circ d(R^{H}_h)_{e_H}(X).
\end{equation*}

Since the derivative $\hat d$ is defined directly from the usual derivative, it shares most of its properties, like linearity, the chain rule, the product rule and so on.

\begin{remark}
    Take $G=\mathbb{R}^n$. Recall that its Lie algebra, formed by the right-invariant vector fields, is  $\mathbb{R}^n$ since we define
    $$X(g)=\left.\frac{d}{dt}\right|_{t=0}(g+tX)\in T_g\mathbb{R}^n$$
    for all $X,g\in\mathbb{R}^n$. Hence, if $c:(-\varepsilon,\varepsilon)\rightarrow \mathbb{R}^n$ is a differentiable curve with $c(0)=g$, and $c'(0)\in T_g\mathbb{R}^n$ denotes the tangent vector in the sense of manifolds, then
    $\phi^{-1}(c'(0))$ $=$ $\lim_{t\rightarrow 0}\frac{c(t)-c(0)}{t}\in\mathbb{R}^n$
    and, thus, the derivative in the classical sense can be recovered using $\phi^{-1}$.

    If $f:O\subset\mathbb{R}^n\rightarrow\mathbb{R}^m$ is a differentiable function then, for each $g,X\in\mathbb{R}^n$,
    $$\hat df_g(X)=\phi_{f(g)}^{-1}df_g\left(\left.\frac{d}{dt}\right|_{t=0}g+tX\right)=\lim_{t\rightarrow 0}\frac{f(g+tX)-f(g)}{t},$$
    and $\hat df_g$ coincides with the Fréchet derivative of $f$. 
\end{remark}

The following constructions are useful  to introduce the concept of ad-rank condition for discrete-time linear control system.

Take $\psi:G\rightarrow H$   an arbitrary automorphism between the Lie groups $G$ and $H$ and $\bpsi:\mathfrak{g}\rightarrow\mathfrak{h}$  its   infinitesimal automorphism. For any $X\in\mathfrak{g}$ and $g\in G$ we have that
$$
    \hat d\psi_g(X)=\phi_{\psi(g)}^{-1}\left.\frac{d}{dt}\right|_{t=0}\psi(e^{tX}g)
        =\bpsi(X). 
$$

Thus, $\hat d\psi_g=\bpsi$ for all $g\in G$.
Now, if  $g\in G$ and
$
         R_g:G\longrightarrow G$ with $
        h\mapsto gh,
  $ is
the right-translation by $g$, then for any  $X\in\mathfrak{g}$ and $h\in G$, one has $d(R_g)_e(X(h))=X(hg)$. Thus, $\hat d(R_g)_h$ coincides with the identity $I:\mathfrak{g}\rightarrow\mathfrak{g}$ for all $h\in G$.
On the other hand, if
$         L_g:G\longrightarrow G$ with      $ h\mapsto gh $
   the left-translation by $g$. For any $h\in G$ and $X\in \mathfrak{g}$,
$$
    \left.\frac{d}{dt}\right|_{t=0}L_g(e^{tX}h) = \left.\frac{d}{dt}\right|_{t=0}ge^{tX}g^{-1}gh=  Ad(g)(X)(gh).
$$

Hence, $\hat d(L_g)_h=Ad(g)$ for all $h\in G$.

Now we can define the ad-rank condition. Then consider the discrete linear control system (\ref{discreteLinearSystem})
and denote by $\phi: \Z \times G \times \mathcal{U} \longrightarrow G$ its solution, where $\mathcal{U} = \prod_{i \in \N} U$.  
Denote by $\psi$ the automorphism $f_0=\phi(1,\cdot,0)$. The flow satisfies
$    \phi(1,g,u)=$  $\phi(1,e,u)\psi(g)$,    and, more generally,
$   \phi(n,g,u)=$  $\phi(n,e,u)\psi^n(g)$. 
Take  $\bpsi:\mathfrak{g}\rightarrow\mathfrak{g}$   the infinitesimal automorphism given by $\psi$ and define the map $F:U\longrightarrow G$  as $F(x)=f_x(e)=\phi(1,e,x)$.

\begin{definition}
    \label{adrankdef} The linear control system (\ref{discreteLinearSystem}) is said to satisfy the ad-rank condition if the smallest $\bpsi$ invariant subspace containing the image of $\hat dF_0$ coincides with $\mathfrak{g}$
\end{definition}

\noindent{\bf $U$ - connected open neighborhood of the origin}: First we assume this condition on $U$ and that $F$ is differentiable in $U$ to prove the accessibility criteria. Later we will show how it can be applied in the case where $U$ is a compact, convex neighborhood of $0$ instead (see Theorem \ref{howToApplyToCompact}).

Let $W\subset \mathfrak{g}$ be the smallest subspace containing the images $\hat dF_x(\mathbb{R}^m)$ for $x\in U$. Equivalently, $W$ is the subspace spanned by the set
$$\{\hat dF_x(y);x\in U,y\in\mathbb{R}^m\}.$$

Consider $V\subset\mathfrak{g}$ the smallest $\bpsi$ invariant subspace containing $W$. If $B$ is a matrix whose columns form a basis of $W$, we have that this space coincides with the image of the Kalman matrix
\begin{equation*}
\begin{pmatrix}
    \bpsi^{n-1}B & \bpsi^{n-2}B & ... & B
\end{pmatrix}
\end{equation*}
where $n=\dim\mathfrak{g}$. Let $\mathfrak{h}$ the Lie sub-algebra generated by $V$. Note that $\mathfrak{h}$ is also $\bpsi$ invariant. In fact, $\bpsi(\mathfrak{h})$ is the sub-algebra generated by $\bpsi(V)$, since $\bpsi$ is an automorphism. However, $\bpsi(V)=V$ as $V$. Equivalently, $\mathfrak{h}$ is the smallest $\bpsi$ invariant sub-algebra of $\mathfrak{g}$ containing the image of all $\hat df_x$ for $x\in U$. This subalgebra will be central for our accessibility criterion.

First, the next proposition gives us another way to define the sub-algebra $\mathfrak{h}$, in terms of the image of $F$ instead of its derivatives.

\begin{proposition}
\label{propImgDer}
Let $H$ the Lie group generated by $\mathfrak{h}$. The image of $F$ is contained in $H$.
\end{proposition}
\begin{proof}
Let $x\in U$ be arbitrary. Writing $r=\dim(H)$ and $n=r+d=\dim(G)$, one can choose a local chart
$$\phi:V_1\subset G\rightarrow V_2\times V_3\subset\mathbb{R}^r\times\mathbb{R}^d$$
such that $x\in V_1$, $\phi(F(x))=0$ and, for each $y\in V_3$, $\phi^{-1}(V_2\times y)$ is entirely contained in some cosset $Hg$. Since $\hat dF_y(X)\in\mathfrak{h}$ for all $y\in U,X\in\mathbb{R}^r$, then $d(\phi\circ F)_y(X)\in\mathbb{R}^r\times 0$ for all $y\in F^{-1}(V_1)$. Thus, there is a neighborhood $O\subset U$ of $x$ and $y\in\mathbb{R}^d$ such that $\phi\circ F(O)\subset V_2\times y$, which implies that $F(O)$ is contained in a some cosset $Hg$. Thus $x\in \hbox{int}(F^{-1}(Hg))$, and, repeating this argument for all $x\in U$ we conclude that the sets $F^{-1}(Hg)$ with $g\in G$ are open and form a partition of $U$. Since $U$ is connected, this implies that the image of $F$ must be entirely contained in one cosset $Hg_0$. Furthermore, $e=F(0)$ and, thus, $e\in Hg_0$, which implies that $Hg_0=H$. Thus, $H$ contains the image of $F$.
\end{proof}

Now let $H^*$ be the smallest subgroup containing the image of $F$. This subgroup is path-connected, and, therefore, a Lie sub-group. Let $\mathfrak{h}^*\subset\mathfrak{g}$ the Lie sub-algebra associated to it, and $\mathfrak{h}_2$ the smallest $\bpsi$ invariant sub-algebra containing $\mathfrak{h}^*$. By the previous proposition, $H^*\subset H$ and, then $\mathfrak{h}^*\subset\mathfrak{h}$. Thus, $\mathfrak{h}_2\subset\mathfrak{h}$. On the other hand, since the image of $F$ is contained in $H^*$ then the image of $\hat dF_x$ is contained in $\mathfrak{h}^*$ for all $x\in U$. Consequently, $\mathfrak{h}\subset\mathfrak{h}_2$, and these two sub-algebras coincide.

Therefore, $\mathfrak{h}$ can equivalently be defined as the smallest $\bpsi$ invariant Lie algebra containing $\mathfrak{h}^*$, where $\mathfrak{h}^*$ is the Lie sub-algebra associated to the smallest sub-group containing the image of $F$.

The next theorem states our accessibility criteria, using the subalgebra $\mathfrak{h}$:

\begin{theorem}
\label{accessibilityCondition}
The system is accessible if, and only if, $\mathfrak{h}=\mathfrak{g}$.
\end{theorem}

Before proving Theorem \ref{accessibilityCondition} we will prove several lemmas and propositions. As in the previous proposition, we will denote by $H\subset G$ the subgroup generated by $\mathfrak{h}$.

\begin{proposition}
\label{proposition2}
For all $k\in\mathbb{N}$, $\mathcal{R}_k$ is contained in $H$.
\end{proposition}
\begin{proof}
First note that $\mathcal{R}_1$ is contained in $H$. In fact, $\mathcal{R}_1$ is the set of all
$\phi(1,e,u)=F(u(0)),$
thus $\mathcal{R}_1$ is the image of $F$. By the previous proposition, this is contained in $H$.

We now prove the proposition by induction. Since $\mathcal{R}_0=\{e\}\subset H$, it is trivial for $\mathcal{R}_0$. Assume that $\mathcal{R}_k\subset H$. Then,
$\mathcal{R}_{k+1}=\mathcal{R}_1\psi(\mathcal{R}_k).$

However, $H$ is invariant by $\psi$ as it is generated by the $\bpsi$ invariant algebra $\mathfrak{h}$, thus, $\psi(\mathcal{R}_k)\subset H$. By the previous argument, $\mathcal{R}_1$ is also contained in $H$, therefore their product must be contained in $H$.
\end{proof}

The proposition above shows that the equality $\mathfrak{h}=\mathfrak{g}$ is a necessary condition for accessibility of the system. In fact, if $\mathfrak{h}\neq\mathfrak{g}$ then $H$ is a lower dimension subgroup and has empty interior in $G$, therefore $\mathcal{R}\subset H$ also has empty interior in $G$.

The next definition is aimed in proving the other implication. For each $k\in\mathbb{N}$ let $\mathcal{F}(k,d)$ the set of diferentiable functions
$f:\mathbb{R}^d\rightarrow G$
such that $f(\mathbb{R}^d)\subset\mathcal{R}_k$. Let
$$\mathcal{F}_k=\bigcup_{d\in\mathbb{N}}\mathcal{F}(k,d) \,\,\,\, \mbox{and} \,\,\,\, \mathcal{F}=\bigcup_{k\in\mathbb{N}}\mathcal{F}_k.$$

Let $W\subset\mathfrak{g}$ a subspace such that $W$ is the image of $\hat df_0$ for some $f\in\mathcal{F}$ and $\dim(W)$ is maximal with this property. By the previous proposition $\mathcal{R}_k\subset H$ for all $k\in\mathbb{N}$, thus $W\subset\mathfrak{h}$. Our goal is to prove the equality.

For the next results, fix an element $f\in\mathcal{F}(k,d)$ such that $\hat df_0=W$.
\begin{lemma}
$W$ is invariant by $\bpsi$ and by $Ad(f(0))$.
\end{lemma}
\begin{proof}
Recall that $\phi(g,0,1)=\psi(g)$ for all $g\in G$. Thus, $\psi\mathcal{R}_k\subset\mathcal{R}_{k+1}$ for all $k\in\mathbb{N}$. Considering the maps
$$\tilde{f}:\mathbb{R}^d\times\mathbb{R}^d\rightarrow G, \tilde{f}(x,y)=f(x)\psi^kf(y), \mbox{ and } \tilde{f}_1:\mathbb{R}^d\times\mathbb{R}^d\rightarrow G \mbox{ with }
\tilde{f}_1(x,y)=f(x)\psi^{k+1}(f(y)),$$
then
$\tilde{f}(x,y)\in\mathcal{R}_k\psi^k\mathcal{R}_k=\mathcal{R}_{2k}$ and  $\tilde{f}_1(x,y)\in\mathcal{R}_k\psi^k\psi\mathcal{R}_k\subset\mathcal{R}_{2k+1}$
for all $(x,y)$, thus $\tilde{f},\tilde{f}_1\in\mathcal{F}$.

Note that for $x,y\in\mathbb{R}^d$,
$\tilde{f}(x,0)=f(x)g=R_g(f(x))$
where $g=\psi^kf(0)$, and
$\tilde{f}(0,y)=L_{f(0)}\circ\psi^k (f(y)).$

Thus, for all $x,y\in\mathbb{R}^n$
$\hat d(\tilde{f})_0(x,0)=\hat df_0(x)$
and
$\hat d(\tilde{f})_0(0,y)=Ad(f(0))\circ\bpsi^k\hat df_0(y).$

Similarly,
$\hat d(\tilde{f}_1)_0(x,0)=\hat df_0(x)$
and
$\hat d(\tilde{f}_1)_0(0,y)=Ad(f(0))\circ\bpsi^{k+1}\hat df_0(y).$

Therefore, the image of $\hat{d}(\tilde{f})_0$ coincides with the sum $W+V$ where
$V=Ad(f(0))\circ\bpsi^k(W).$
Since $\dim(W)$ is maximal, then $\dim(W+V)\le\dim(W)$, which implies $W+V=W$. Furthermore $\dim(V)=\dim(W)$ since the linear transformations above are all automorphisms, therefore $V=W$. Then,
$Ad(f(0))\circ\bpsi^k(W)=W.$

Applying a similar argument to $\tilde{f}_1$ one also concludes that
$Ad(f(0))\circ\bpsi^{k+1}(W)=W.$

Then
$Ad(f(0))\circ\bpsi^{k+1}(W)=Ad(f(0))\circ\bpsi^k(W)$ and hence
$\Rightarrow \bpsi(W)=W,$
since $Ad(f(0))$ and $\bpsi$ are invertible, which shows that $W$ is invariant by $\bpsi$. Also as a consequence, we have
$$W=Ad(f(0))\circ\bpsi^k(W)=Ad(f(0))(W).$$
\end{proof}

\begin{lemma}
$W$ is invariant by $Ad(g)$ for all $g\in\mathcal{R}$.
\end{lemma}
\begin{proof}
Consider $g\in\mathcal{R}$. Then $g\in\mathcal{R}_l$ for some $l\in\mathbb{N}$. Define
$f_1:\mathbb{R}^d\times\mathbb{R}^d\rightarrow G$ bu $f_1(x,y)= f(x)\psi^k(g\psi^l(f(y))).$

Hence $f_1\in\mathcal{R}_{2k+l}$ and, for $x,y\in\mathbb{R}^d$,
$\hat d(f_1)_0(x,0)=$ $\hat df_0(x)$ and
$\hat d(f_1)_0(0,y)=$ $Ad(f(0))\circ \bpsi^k\circ Ad(g)\circ\bpsi^l\circ \hat df_0(y).$

We can conclude that the image of $d(f_1)_0$ is the vector subspace $W+V$ where 
$V=$ $Ad(f(0))\circ \bpsi^k\circ Ad(g)\circ\bpsi^l(W).$

Therefore, we get $V=W$ and
$W=$ $V=$ $Ad(f(0))\circ \bpsi^k\circ Ad(g)\circ\bpsi^l(W).$
Since $W$ is invariant by $Ad(f(0))$ and $\bpsi$, it is also invariant by the inverse of these automorphisms, we obtain that
$W=$ $\bpsi^{-k}\circ Ad(f(0))^{-1}(W)=Ad(g)\circ\bpsi^l(W)=$ $Ad(g)(W).$
\end{proof}
\begin{lemma}
$W$ is a subalgebra.
\end{lemma}
\begin{proof}
For each $X\in\mathfrak{g}$ consider the function
$F_X:G\rightarrow \mathfrak{g}$ ,
$g\mapsto Ad(g)(X).$

Taking $\mathfrak{g}\simeq\mathbb{R}^n$ as its own Lie algebra, for all $Y\in\mathfrak{g}$ and $g\in G$ we have
$$\hat d(F_X)_g(Y)=\left.\frac{d}{dt}\right|_{t=0}Ad(e^{tY}g)(X)=
\left.\frac{d}{dt}\right|_{t=0}e^{ad(tY)}Ad(g)(X)=ad(Y)(Ad(g)(X)).$$

Now let $X,Y\in W$, $y\in\mathbb{R}^d$ such that $\hat df_0(y)=Y$, and $Z=Ad(f(0))^{-1}(X)$. Note that $Z\in W$ as $W$ is invariant by $Ad(f(0))$. Consider the curve
$c:\mathbb{R}\rightarrow \mathfrak{g}$ with
$c(t)=F_Z(f(ty))=Ad(f(ty))(Z).$

The image of $c$ is contained in $W$ since $f(ty)\in\mathcal{R}$ for all $t\in\mathbb{R}$. Thus, the image of $c'$ must also be contained in $W$. Furthermore,
$c'(0)=$ $\hat d(F_Z)_{f(0)}\hat df_0(y)=$ $ad(Y)(Ad(f(0))(Z))=ad(Y)(X).$

Consequently, we get $ad(Y)(X)\in W$. Since $X,Y\in W$ are arbitrary, we conclude that $W$ is a subalgebra. 
\end{proof}

\begin{lemma}
For all $\tilde f\in\mathcal{F}$, $W$ contains the image of $\hat d(\tilde f)_0$.
\end{lemma}
\begin{proof}
Let $\tilde f\in\mathcal{F}(l,r)$. Define
$f_2:\mathbb{R}^r\times\mathbb{R}^d$
by $f_{2}(x,y)= \tilde f(x)\psi^{l}(f(y)).$

Then, for $x\in\mathbb{R}^r$ and $y\in\mathbb{R}^d$ we get that 
$\hat d(f_2)_0(x,0)=$ $\hat d(\tilde f)(x)$ and $\hat d(f_2)_0(0,y)=$ $Ad(\tilde f(0))\bpsi^l\hat df_0(y).$

If $V$ denotes the image of $\hat d(\tilde f)_0$, we obtain that the image of $\hat d(f_2)_0$ coincides with the sum
$V+Ad(\tilde f(0))\bpsi^d(W)=$ $V+W.$
Since the dimension of $W$ is maximal, $V$ must be contained in $W$.
\end{proof}

\begin{lemma}
For all $x\in U$ and $y\in\mathbb{R}^m$, $W$ contains  $\hat dF_x(y)$.
\end{lemma}
\begin{proof}
Consider the curve
$c:\mathbb{R}\rightarrow G$, given by 
$c(t)=$ $F(x+\alpha(t)y)$
where $\alpha$ is a diffeomorfism from $\mathbb{R}$ to $(-\varepsilon,\varepsilon)$ and $\varepsilon$ is sufficiently small so that $x+(-\varepsilon,\varepsilon)y\in U$. Then the image of $c$ is contained in $\mathcal{R}_1$ and
$\hat d c_0=$ $\alpha'(0)\hat dF_x(y).$
$\alpha'(0)$ is nonzero since $\alpha$ is a diffeomorphism, therefore the image of $\hat dc_0$ is the subspace spanned by $\hat dF_x(y)$. By the previous proposition, $W$ contains this subspace and, therefore, contains $\hat dF_x(y)$.
\end{proof}

The previous lemmas show that $W$ is invariant by $\bpsi$, is a sub-álgebra, and contains the image of $B$. Therefore, $\mathfrak{h}\subset W$. It was seen that the other inclusion must also be true, therefore $W=\mathfrak{h}$. Consequently, there exists $f\in\mathcal{F}$ such that $\mathfrak{h}$ is the image of $\hat df_0$.

We can now prove Theorem \ref{accessibilityCondition}.

\begin{proof}[Proof \textbf{(Theorem \ref{accessibilityCondition})}:]
It was seen previously that the equality $\mathfrak{h}=\mathfrak{g}$ is necessary for accessibility. For the other implication, assume that $\mathfrak{h}=\mathfrak{g}$. Then, from the previous lemmas, there is $f\in\mathcal{F}_{k,d}$ such that $\hat df_0$ is surjective. Consequently, by the submersion Theorem, $f(0)$ is in the interior of the image of $f$, and, then,
$f(0)\in$ $Int(f(\mathbb{R}^d))\subset$ $Int(\mathcal{R}_k).$
Therefore, $Int(\mathcal{R}_k)$ is nonempty.
\end{proof}

\noindent{\bf $U$ -  compact and convex containing the origin in its interior}. Now we assume this condition to prove the accessibility criteria. The next lemma allows us to apply the previous results such $U$.

\begin{lemma}
Let $c:[0,1]\rightarrow G$ a continuous curve on a Lie group $G$, $H\subset G$ a connected Lie subgroup and assume that $c(t)\in H$ for all $t<1$. Then $c(1)\in H$.
\end{lemma}
\begin{proof}
Let $r=\dim(H)$ and $n=r+d=\dim(G)$. Then there is a local chart
$\phi:V_1\subset G\rightarrow V_2\subset\mathbb{R}^r\times\mathbb{R}^d$
centered in $c(1)$ ($\phi(c(1))=0$) such that, for each $y\in\mathbb{R}^d$, the inverse image $\phi^{-1}(\mathbb{R}^r\times\{y\})$ is an open set of some cosset $Hg$. Thus, if $(x,y)\in V_2$ is such that $\phi^{-1}(x,y)\in H$ then $\phi^{-1}(\mathbb{R}^r\times y)$ is open and closed in $H\cap V_1$.

Since $c(1)\in V_1$ and $c$ is continuous, there is $\varepsilon>0$ such that $c(1-\varepsilon,1)\subset V_1$. Since $(1-\varepsilon,1)$ is connected, the previous argument implies that $c(1-\varepsilon,1)$ is entirely contained in a set $\phi^{-1}(\mathbb{R}^r\times y)$ for some $y\in\mathbb{R}^d$, and, since $\phi(c(1))=0$, then $y=0$, which implies that $c(1)\in H$.
\end{proof}
\begin{corollary}
The smallest subgroup containing $F(U)$ coincides with the smallest subgroup containing $F(Int(U))$.
\end{corollary}
\begin{proof}
Since $U$ is convex and has nonempty interior, for each $x$ in the boundary $\partial(U)$ there is a curve continuous curve $c:[0,1]\rightarrow U$ such that $c(1)=x$ and $c(t)\in Int(U)$ for all $t<1$. By the previous proposition, $F(x)$ is contained in the smallest sub-group containing $F(Int(U))$.
\end{proof}

Then, for the case where $U$ is convex and compact containing the origin in the interior, we must assume that $F$ is differentiable in $Int(U)$ and continuous in $U$. In this case, we define $\mathfrak{h}\subset\mathfrak{g}$ as the smallest $\bpsi$ invariant sub-algebra containing the images $\hat dF_x(\mathbb{R}^m)$ for $x\in Int(U)$. By the above proposition combined with a previous argument, $\mathfrak{h}$ is equivalently defined as the smallest $\bpsi$ invariant sub-algebra containing $\mathfrak{h}^*$, where $\mathfrak{h}^*$ is the sub-algebra associated to the smallest subgroup containing the image of $F$. In this case we have a similar result:

\begin{theorem}
\label{howToApplyToCompact}
The system is accessible if, and only if, $\mathfrak{h}=\mathfrak{g}$.
\end{theorem}
\begin{proof}
A similar argument to Proposition \ref{proposition2} shows that $\mathcal{R}$ is contained in the subgroup $H$ generated by $\mathfrak{h}$, therefore, if $\mathfrak{h}\neq\mathfrak{g}$ then the system is not accessible. If $\mathfrak{h}=\mathfrak{g}$ then we can consider the system with controls restricted to $Int(U)$. This system is accessible by the previous arguments, and, therefore, the original system is also accessible.
\end{proof}

%\begin{example}
%Consider the discrete linear system defined by

%$$x_{k+1}=\exp(Bu_k)\psi(x_k),k\in\N_0$$
%where $\psi$ is an automorphism in $G$, $B:\mathbb{R}^m\rightarrow \mathfrak{g}$ is a linear transformation and $u_k\in U\subset \mathbb{R}^m$ is a compact and convex set containing the origin in its interior. $\phi$ is defined recursively for $n\in\mathbb{N}$ by
%$$\phi(x,u,n+1)=\phi(\phi(x,u,n),n\cdot u,1)$$
%where $n\cdot u$ is the shift defined by $n\cdot u(k)=u(n+k)$.

%In this case, it can be shown that $\mathfrak{h}$ coincides with the smallest $\bpsi$ invariant sub-algebra containing the image of $B$.
%\end{example}

The next result shows a version of the ad-rank Theorem (see Theorem 3.5 in \cite{AyalaAndTirao}), adapted for the discrete case.   

\begin{theorem}\label{adrank} (Ad-rank) If the system (\ref{discreteLinearSystem}) satisfies the ad-rank condition, then $e \in \hbox{int}(\mathcal{R}_n)$, where $n=\dim(\mathfrak{g})$. In particular, the system is locally controllable in $e$.
\end{theorem}
\begin{proof}
Let $V$ be the smallest $\bpsi$ invariant subspace containing the image of $\hat dF_0$. Then $V$ is spanned by the set $\{\bpsi^i\circ\hat dF_0(x);x\in\mathbb{R}^m,i\in\{0,1,\ldots,n-1\}\}$.

Define the map
$G:U^n\rightarrow G$ by
\begin{equation*}
    G(x_0,x_1, \ldots,x_{n-1})=\varphi(e,u_{\mathbf{x}},n)=F(x_0)\psi(F(x_1))\psi^2(F(x_2))\cdots\psi^{n-1}(F(x_{n-1})), 
\end{equation*}
where $u_{\mathbf{x}}$ is defined by
$$u_{\mathbf{x}}(k)=\left\{\begin{array}{l}
x_k;\text{ if }k\in\{0,1,2, \ldots,n-1\}\\
0;\text{ otherwise}
\end{array}\right.$$

Then $G(0,0,\ldots,0)=e$, and the image of $G$ is contained in $\mathcal{R}_n$. Furthermore, if $(0,0,\ldots,x_i,\ldots,0)$ is a vector in $(\mathbb{R}^m)^n$ (each coordinate is a vector of $\mathbb{R}^m$) where the nonzero entry $x_i$ is the $i$-th entry, then
$\hat dG_{(0,0,\ldots,0)}(0,0,\ldots,x_i,\ldots,0)=$ $\bpsi^i\circ\hat dF_0(x_i).$
Thus, the image of $\hat dG_{(0,0,\ldots,0)}$ coincides with $W$. If the hypothesis $W=\mathfrak{g}$ is satisfied, then by the submersion theorem $e=G(0,0,\ldots,0)\in int(img(G))\subset int(\mathcal{R}_n)$.
\end{proof}

If we denote
$$X_i=\frac{\partial}{\partial x_i}F(0)=\left.\frac{d}{dt}\right|_{t=0}\phi(e,te_i,1)$$
where $e_i\in\mathbb{R}^m$ denotes the vector $(0,0,\ldots,1,\ldots,0)$ where only the $i-th$ coordinate is nonzero, then the space from the Theorem above coincides with
$${\rm span}\{\bpsi^k(X_i);i\in\{1,2,\ldots,m\},k\in\{0,1,\ldots,n-1\}\}.$$

\begin{remark}
    Since $F(0)=e$, then the image of $\hat dF_0$ coincides with the image of $dF_0$ if one considers $\mathfrak{g}=T_eG$. Thus, the Theorem above can equivalently be stated as: If the smallest $\bpsi$ invariant subspace containing the image of $dF_0$ coincides with $\mathfrak{g}$ then $e\in \hbox{int}(\mathcal{R}_n)$.
\end{remark}

The hypothesis of the previous Lemma can be calculated as follows: the smallest $\bpsi$ invariant subspace containing the image of $\hat dF_0$ coincides with the image of the matrix
$$
    \mathcal{V} = \begin{pmatrix}
    \bpsi^{n-1}\hat dF_0 & \bpsi^{n-2}\hat dF_0 & \cdots & \hat dF_0
\end{pmatrix}.$$
Thus, if the matrix above has full rank then the system is locally controllable in $e$.

Finally, to ensure the importance of the ad-rank condition for discrete-time linear systems, we have the following criteria for the existence of control sets.

\begin{proposition}Consider a linear system (\ref{discreteLinearSystem}) defined over the Lie group $G$ satisfying the ad-rank condition. Then, there exists a control set $D$, with non-empty interior such that $D = \overline{\mathcal{R}}\cap \mathcal{C}$. 
\end{proposition}

\begin{proof} If $\mathcal{R}$ is open, by the Remark (\ref{openess1}), $\mathcal{C}$ is open. Now, as $e \in \mathcal{R} \cap \mathcal{C}$, there is a neighborhood $V$ of $e$ such that $e \in V \subset \mathcal{R} \cap \mathcal{C}$. Then, the neighborhood $V$ is completely controllable. This proves that $V$ satisfies the properties $1$ and $2$ of the definition of control sets. Hence there is a control set $D$ such that $V \subset D$. As $V$ is open, $e \in \hbox{int}D$. This also proves that $e \in \hbox{int}\mathcal{C}$. As $\mathcal{R}$ is open, the system (\ref{discreteLinearSystem}) is accessible. By \cite[Proposition 6]{CCS2}, $D$ has the form $D = \mathcal{C} \cap \overline{\mathcal{R}}$. 
    
\end{proof}

\subsection{Examples}

\begin{example} Take  $G = {\rm Sl}(2,\R)$ the semisimple connected Lie group of $2 \times 2$ matrices with real entries  and determinant $1$. Denote its Lie algebra by $\mathfrak{sl}(2,\R)$. 
It is known that $\hbox{Aut}(\mathfrak{sl}(2,\R)) = \hbox{Inn}(\mathfrak{sl}(2,\R))$, that is, every automorphism of $\mathfrak{sl}(2,\R)$ is inner in the sense of if $T \in \hbox{Aut}(\mathfrak{sl}(2,\R))$, there are $Y_1, \ldots,Y_n \in \mathfrak{gl}(2,\R)$ such that 
$    T(X) = e^{\adj{Y_1}} \cdots e^{\adj{Y_{n}}}(X)$. By \cite[Proposition 5.15]{sanmartin1}, we have that the conjugation $C_h(g) = h g h^{-1}$ has as differential at the identity the function $T$ above, where $h = e^{Y_1} \cdots e^{Y_n} \in {\rm Gl}(2,\R)$.

Then we can define the class of linear systems of ${\rm Sl}(2,\R)$. In fact, given a $h \in {\rm Gl}(2,\R)$, consider a map $f: U \times {\rm Sl}(2,\R) \longrightarrow {\rm Sl}(2,\R)$ given by 
\begin{equation*}
    f_u(g) = 
    \begin{bmatrix}
        f^{11}_u(e) & f^{12}_u(e)\\
        f^{21}_u(e) & f^{22}_u(e)
    \end{bmatrix} hgh^{-1}
\end{equation*}
such that $f^{11}_0(e) = f^{22}_0(e) = 1$, $f^{21}_0(e) = f^{12}_0(e) = 0$ and $f^{11}_u(e)f^{22}_u(e) - f^{21}_u(e)f^{12}_u(e)=1$ for all $u \in U$. Considering the discrete-time system
\begin{equation} \label{discretelinearSL}
     x_{k+1} = f_{u_k}(x_k), k \in \N_0,
\end{equation}
it is not hard to prove that the system (\ref{discretelinearSL})  is a linear system on ${\rm Sl}(2,\R)$. In particular, every linear system in ${\rm Sl}(2,\R)$ has the form of the system (\ref{discretelinearSL}) given the product property in the definition of linear systems in the introduction.

Now, take the matrices $h, h^{-1} \in {\rm Gl}(2,\R)$ as
\begin{equation*}
    h= 
    \begin{bmatrix}
    h_{11} & h_{12}\\
    h_{21} & h_{22}
    \end{bmatrix} \hbox{ and }
    h^{-1} = \frac{1}{h_{11}h_{22} - h_{21}h_{12}}
    \begin{bmatrix}
    h_{22} & -h_{21}\\
    -h_{12} & h_{11}
    \end{bmatrix}, 
\end{equation*}
and a element $g \in {\rm Sl}(2,\R)$ in the form $g = \begin{bmatrix}
    g_{11} & g_{12}\\
    g_{21} & g_{22}
    \end{bmatrix}$. We have 
\begin{equation*}
    f_u(g) = 
    \begin{bmatrix}
        f^{11}_u(e) & f^{12}_u(e)\\
        f^{21}_u(e) & f^{22}_u(e)
    \end{bmatrix} 
    f_0(g),
\end{equation*}
where $f_0(g) = h g h^{-1}$ is given by
\begin{equation*}
     f_0(g) = 
    \begin{bmatrix}
       \frac{g_{12} h_{11} h_{21} + g_{22} h_{12} h_{21} - g_{11} h_{11} h_{22} - g_{21} h_{12} h_{22}}{h_{12} h_{21} - h_{11} h_{22}} & \frac{(g_{12} h_{11}^2 - h_{12} (g_{11} h_{11} - g_{22} h_{11} + g_{21} h_{12}))}{(-h_{12} h_{21} + h_{11} h_{22})}\\ 
       \frac{(-g_{12} h_{21}^2 + h_{22} (g_{11} h_{21} - g_{22} h_{21} + g_{21} h_{22}))}{(-h_{12} h_{21} + h_{11} h_{22})} & \frac{(g_{12} h_{11} h_{21} + g_{22} h_{11} h_{22} - h_{12} (g_{11} h_{21} + g_{21} h_{22}))}{(-h_{12} h_{21} +h_{11} h_{22})}
    \end{bmatrix}
\end{equation*}  

Also, considering the function $f_0$ as a function in the form $f_0: \R^4 \longrightarrow \R^4$, we have that  
\begin{equation*}
    df_0 = 
    \left[
\begin{array}{cccc}
    -\frac{h_{11} h_{22}}{h_{12} h_{21}-h_{11} h_{22}} & \frac{h_{11} h_{21}}{h_{12} h_{21}-h_{11} h_{22}} & -\frac{h_{12} h_{22}}{h_{12} h_{21}-h_{11} h_{22}} & \frac{h_{12} h_{21}}{h_{12} h_{21}-h_{11} h_{22}} \\
    -\frac{h_{11} h_{12}}{h_{11} h_{22}-h_{12} h_{21}} & \frac{h_{11}^2}{h_{11} h_{22}-h_{12} h_{21}} & -\frac{h_{12}^2}{h_{11} h_{22}-h_{12} h_{21}} & \frac{h_{11} h_{12}}{h_{11} h_{22}-h_{12} h_{21}} \\
    \frac{h_{21} h_{22}}{h_{11} h_{22}-h_{12} h_{21}} & -\frac{h_{21}^2}{h_{11} h_{22}-h_{12} h_{21}} & \frac{h_{22}^2}{h_{11} h_{22}-h_{12} h_{21}} & -\frac{h_{21} h_{22}}{h_{11} h_{22}-h_{12} h_{21}} \\
    -\frac{h_{12} h_{21}}{h_{11} h_{22}-h_{12} h_{21}} & \frac{h_{11} h_{21}}{h_{11} h_{22}-h_{12} h_{21}} & -\frac{h_{12} h_{22}}{h_{11} h_{22}-h_{12} h_{21}} & \frac{h_{11} h_{22}}{h_{11} h_{22}-h_{12} h_{21}} \\
\end{array}\right]
\end{equation*}

Let us  explore some numeric examples. Take the case when $U \subset \R$ is a compact convex neighborhood of $0$ and 
\begin{equation}\label{matrixh}
    h = 
    \begin{bmatrix}
        1&1\\
        0&1
    \end{bmatrix}, 
\end{equation}
with 
\begin{equation}\label{fu}
    f_u(e) = 
    \begin{bmatrix}
        1+u&-u\\
        u&1-u
    \end{bmatrix}.
\end{equation}

Let be 
\begin{equation} \label{discretelinearSL2}
     g_{k+1} = f_{u_k}(g_k), k \in \N_0,
\end{equation}
the linear system defined by the function above. Considering the $2-$step nilpotent matrix
\begin{equation*}
    M = \begin{bmatrix}
        0 & 1 \\
        0 & 0
    \end{bmatrix}. 
\end{equation*}
we get that 
\begin{equation}
    e^M = \sum_{ n \in \N} \frac{M^n}{n!} = I + M = 
    \begin{bmatrix}
        1&1\\
        0&1
    \end{bmatrix} = h 
\end{equation}

At first, the function $f_0$ is defined by
\begin{equation*}
    f_0(g_{11},g_{12}, g_{21}, g_{22}) = \left(g_{11}+g_{21} ,-g_{11}+g_{12}-g_{21}+g_{22}, g_{21}, g_{22}-g_{21}
    \right)
\end{equation*}
with derivative given by 
\begin{equation}\label{df_0}
    df_0 = \left[
    \begin{array}{cccc}
        1 & 0 & 1 & 0 \\
        -1 & 1 & -1 & 1 \\
        0 & 0 & 1 & 0 \\
        0 & 0 & -1 & 1 \\
    \end{array}\right]
\end{equation} 

Let us check the accessibility of the system (\ref{discretelinearSL2}). It is simple to verify that $f_u(e) \in {\rm Sl}(2,\R),$ for every $u \in U$. Following the definition (\ref{regular}), we claim that $e \in \hat{\mathcal{R}}$. In fact, take $(u,v,w) \in \hbox{int}U^3$. For $k=3$, using the notation $f_{u,v,w} = f_{u} \circ f_{v} \circ f_{w}$, we have 
\begin{equation*}
    f_{u,v,w}(e) = f_{u}(e) C_h (f_{v}(e)) C_{h^2}(f_{w}(e))
\end{equation*}

The matrix above is given by

\resizebox{\linewidth}{!}{%
$\displaystyle
    f_{u,v,w}(e)=\left[    
    \begin{array}{cc} 
        (w+1)((u+1) (2 v+1)-u v)&  (-u (1 - 2 v) - 4 (1 + u) v) + (-u v + (1 + u) (1 + 2 v))\\
        ((1 - u) v + u (1 + 2 v)) (1 + w) & ((1 - u) (1 - 2 v) - 4 u v) + ((1 - u) v + u (1 + 2 v))
    \end{array} \right]
    $}
where $f_{u,v,w} = f_{u} \circ f_{v} \circ f_{w}$, whose derivative is given by 
\begin{equation}
    \frac{\partial }{\partial(u,v,w)} f_{u} \circ f_{v} \circ f_{w}(e) =  
    \begin{bmatrix}
        \begin{bmatrix}
            (1+v)(1+w)\\
            (2 + u) (1 + w)\\
            1+ 2v + u (1+ v)
        \end{bmatrix} &
        \begin{bmatrix}
            -v\\ -2-u\\
            0
        \end{bmatrix}\\
        \begin{bmatrix}
            (1+v)(1+w)\\ 
            (1+u)(1+w)\\ 
            v + u (1+v)
        \end{bmatrix} &
        \begin{bmatrix}
           -v\\ -1-u\\ 0
        \end{bmatrix}
    \end{bmatrix}.
\end{equation}

Taking the vectors of the matrix above, one can prove that the subspace generated by them is $3-$dimensional. Hence, the matrix above has rank $3$, for every $u \in \hbox{int}U$. Then $e \in \hat{\mathcal{R}}_3 \subset \hat{\mathcal{R}}$. As the set $\hat{\mathcal{R}}$ is open and $\hat{\mathcal{R}} \subset \mathcal{R}$, we have $e \in \hbox{int}\mathcal{R}$. By the Proposition (\ref{reachablesetprop}), item 6, the set $\mathcal{R}$ is open, which also implies that the system is accessible. 

Now, we will prove that the subalgebra $\mathfrak{h}$ from Theorem (\ref{accessibilityCondition}) coincides with $\mathfrak{sl}(2,\R)$.  In fact, considering the function $F: U \longrightarrow {\rm Sl}(2,\R)$ defined by $F(u) = f_u(e)$ as in (\ref{fu}), we first need to find the set 
\begin{equation*}
    W = Span\{\hat{d}F_u(X): u \in U, X \in \R\}. 
\end{equation*}
where
\begin{equation*}
    \hat{d}F_u(X) = d(R^{s}_{F(u)^{-1}})_{F(u)} \circ d(F)_u \circ d(R^{r}_u)_e(X). 
\end{equation*}and $R^s$ and $R^r$ are the respective right-translations in the Lie groups ${\rm Sl}(2,\R)$ and $\R$. As a matter of fact, we have 
\begin{equation*}
    d(R^{s}_{X})_{Y}(Z) = ZX. 
\end{equation*}

Also, as the right-invariant vector fields in $\R$ are the constant fields, we get that 
\begin{equation}\label{dF_u}
    \hat{d}F_u(X) = X\begin{bmatrix}
        1 & -1 \\
        1 & -1
    \end{bmatrix}\cdot \begin{bmatrix}
        1-u & u \\
        -u & 1+u
    \end{bmatrix} = X \begin{bmatrix}
        1 & -1 \\
        1 & -1
    \end{bmatrix}. 
\end{equation}

Therefore 
\begin{equation*}
    W = \left\{k\begin{bmatrix}
        1 & -1\\
        1 & -1
    \end{bmatrix}\in \mathfrak{sl}(2,\R): k \in \R\right\}. 
\end{equation*}
whose dimension is 1. Consider the Kalman matrix 
\begin{equation*}
    \mathcal{K} = (df_0^2 X_0 \hbox{ }df_0 X_0 \hbox{ }X_0), 
\end{equation*}
with $X_0 = \begin{bmatrix}
    1 & -1\\
    1 & -1
\end{bmatrix}$ and $df_0$ defined in (\ref{df_0}). We get that 
\begin{equation}\label{Kalmanmatrix}
    \mathcal{K} = \begin{bmatrix}
        3 & 2 & 1\\
        -9 & -4 & -1\\
        1 & 1 & 1\\
        -3 & -2 & -1
    \end{bmatrix}. 
\end{equation}
and using some concepts of linear algebra, it is easy to show that $rank(\mathcal{K}) = 3$. As $\mathfrak{h}$ is the smallest subalgebra containing the image of $\mathcal{K}$ with dimension $3$, we get $\mathfrak{h} = \mathfrak{sl}(2,\R)$ and hence, the system is accessible. 

Now, for the the ad-rank condition from Theorem (\ref{adrank}), the vector $$X_0 = \begin{bmatrix}
    1 & -1\\
    1 & -1
\end{bmatrix}$$ is the matrix of $\hat{d}F_0$ by the expression in (\ref{dF_u}). Therefore, when we generates the matrix $\mathcal{V} = (\bpsi^2 \hat{d}F_0 \hbox{ }\bpsi \hat{d}F_0 \hbox{ }\hat{d}F_0),$ it coincides with the matrix $\mathcal{K}$ in (\ref{Kalmanmatrix}), whose rank is $3$, which is exactly the dimension of $\mathfrak{sl}(2,\R)$, implying local controllability. 
\end{example}

\begin{example}The real abelian solvable Lie groups of dimension $2$ are $\R^2$, $\mathbb{T} \times \R$ and $\mathbb{T}^2$ (see e.g. \cite{onish}). In the non-abelian case, the unique (up to an isomorphism) real solvable lie group is the open half plane $G = \R^{+} \ltimes \R$, endowed with the product 
	\begin{equation*}
		(x_1,y_1) \cdot (x_2,y_2) = (x_1x_2, y_2 + x_2 y_1). 
	\end{equation*}
	
	The Lie group $(G, \cdot)$ is called affine group and denoted by $\hbox{Aff}(2,\R)$. The Lie algebra of $\hbox{Aff}(2,\R)$ is given by the set $\R^2$, endowed with the Lie bracket 
\begin{equation*}
    [(x,y),(z,w)] = (0,xw - yz) 
\end{equation*}
which is known as affine Lie algebra, denoted by $\mathfrak{aff}(2,\R)$. 

In particular, the automorphisms of $\hbox{Aff}(2,\R)$ are given by 
\begin{equation}\label{automorphismsaff}
	\phi(x,y) = (x,a(x-1) + dy),
\end{equation}
with  $ d \in \R\setminus \{0\}$ and $a \in \R$. Then, the linear systems on $\hbox{Aff}(2,\R)$ can be defined by the functions 
\begin{equation}\label{sist2dim}
	f((x,y), u):=f_u(x,y) = (h(u)x, a(x-1) + dy + g(u)x),
\end{equation}
where $h: \R^m \rightarrow \R^+$ and $g: \R^m \rightarrow \R$ are $\mathcal{C}^{\infty}$ maps satisfying $h(0)=1$ and $g(0)=0$. Consider the linear system
\begin{equation}\label{systemOnAff}
	 x_{k+1} = f(x_k, u_k),\ k \in \N,\ u\in U,
\end{equation}
where  $U$ is assumed to be a compact and convex neighborhood of $0\in\mathbb{R}^m$ and $f$ has the form in (\ref{sist2dim}). Considering the function $F: U \longrightarrow \hbox{Aff}(2,\R)$ defined by $F(u) = f_u(1,0)$, we get $F(u) = (h(u), g(u))$. If $R^a_{(x,y)}$ is the right-translation map of $\hbox{Aff}(2,\R)$, we obtain
\begin{equation*}
    R^a_{(x,y)}(z,w) = (z,w)\cdot (x,y) = (zx, y + xw),
\end{equation*}
which also implies
\begin{equation*}
    d(R^a_{(x,y)})_{(z,w)} = 
    \begin{bmatrix}
        x & 0 \\
        0 & x
    \end{bmatrix}. 
\end{equation*}

The fact of $(h(u),g(u))^{-1} = (\frac{1}{h(u)}, -\frac{g(u)}{h(u)})$, we get that
\begin{eqnarray*}
    \hat{d}F_u(X) &=& d(R^a_{F(u)^{-1}})_{F(u) }\circ d(F)_u \circ d(R^r_u)_e(X)\\
    &=& \begin{bmatrix}
        \frac{1}{h(u)} & 0 \\
        0 & \frac{1}{h(u)}
    \end{bmatrix}\cdot \begin{bmatrix}
        Xh'(u) \\
        Xg'(u)
    \end{bmatrix}\\
    &=&X\begin{bmatrix}
        \frac{h'(u)}{h(u)} \\
        \frac{g'(u)}{h(u)}
    \end{bmatrix}
\end{eqnarray*}

Consider the case $h(u) = e^u$, $g(u) = u^2$ and $U = [-1,1]$. Then
\begin{equation*}
    \hat{d}F_u(X) = \begin{bmatrix}
        X \\
        2Xue^{-u}
    \end{bmatrix}
\end{equation*}

Therefore, the vector subspace $W$ can be spanned by the vectors $v_1 = (1,0)$ and $v_2 = (1,2)$, which imples that $\dim W  = 2$. Hence, the Kalman matrix $\mathcal{K} = (df_0 B \hbox{ }B)$ with $B = \begin{bmatrix}
    1 & 1 \\
    0 & 2
\end{bmatrix}$ has rank $2$ and thus, $\mathfrak{h} = \mathfrak{aff}(2,\R)$. One can notice that in this case, there is no need to construct the subalgebra $\mathfrak{h}$, given that $\dim W = \dim \mathfrak{aff}(2,\R)$.

Let us verify the hypothesis for the ad-rank condition. As a matter of fact, considering that $\hat{d}F_0 = \begin{bmatrix}
    1\\
    0
\end{bmatrix}$, we have that 
\begin{equation*}
    \mathcal{V} = (\bpsi \hat{d}F_0 \hbox{ }\hat{d}F_0) = 
    \begin{bmatrix}
       1 & 1 \\
       a & 0
    \end{bmatrix}
\end{equation*}

Thus, the ad-rank condition from theorem (\ref{adrank}) is satisfied if, and only if, $a \neq 0$. As all the automorphism of $\hbox{Aff}(2,\R)$ has the form in (\ref{automorphismsaff}), if we consider $f_0$ as a inner automorphism and $d= 1$, it follows by \cite[Theorem 79]{TMC} that the system $(\ref{systemOnAff})$ is controllable.  
\end{example}

\begin{example}Consider the Heisenberg group 
\begin{equation*}
    \mathbb{H}=\left\{\begin{bmatrix}
			1 & x_2 & x_1\\
			0 & 1 & x_3\\
			0 & 0 & 1
		\end{bmatrix};\ x_1,x_2, x_3\in\mathbb{R}\right\},
\end{equation*}
which is diffeomorphic to the Euclidean space $\mathbb{R}^{3}$ endowed with the product
\begin{equation*}
    (x_{1},x_{2},x_{3})\cdot(y_{1},y_{2},y_{3})=(x_{1}+y_{1}+x_{2}y_{3},x_{2}+y_{2},x_{3}+y_{3}).
\end{equation*}

Also, $\mathbb{H}$ is a Lie group with Lie algebra 
\begin{equation*}
    \mathfrak{g} = \left\{X \in \mathfrak{gl}(3,\R): X = \begin{bmatrix}
        0 & x & y\\
        0 & 0 & z\\
        0 & 0 & 0
    \end{bmatrix}, (x,y,z) \in \R^3\right\}. 
\end{equation*}
endowed with the matrix Lie bracket. 

Let $U$ be a compact and connected neighborhood of $0\in\mathbb{R}$ and $f:\mathbb{H}\times U\rightarrow\mathbb{H}$ given by
\begin{equation*}
    f_{u}(x_{1},x_{2},x_{3})=\left(x_{1}+x_{2}+\dfrac{x_{2}^{2}}{2}+ux_{2}+ux_{3}-\dfrac{u}{2}-\dfrac{u^{2}}{3},x_{2}+u,x_{2}+x_{3}+\dfrac{u}
		{2}\right).
\end{equation*}

Then $f_0: \mathbb{H} \longrightarrow \mathbb{H}$ is defined by 
\begin{equation}\label{f_0H}
    f_0(x_1,x_2,x_3) = \left(x_1+x_2+\frac{x_2^2}{2}, x_2,x_2+x_3\right)
\end{equation}
and $F:U \longrightarrow \mathbb{H}$ by 
\begin{equation*}
    F(u) = \left(-\frac{u}{2} - \frac{u^2}{3}, u,\frac{u}{2}\right)
\end{equation*}

It is not hard to check that 
\begin{eqnarray}\label{heisenberg}
	g_{k+1}=f_{u_k}(g_{k}), \ u_k\in U
\end{eqnarray}
is a linear system on $\mathbb{H}$. Let us prove the accessibility property from Theorem (\ref{howToApplyToCompact}), that is, $\mathfrak{h} = \mathfrak{g}$. At first, the right-translation in $\mathbb{H}$ is given by the function
\begin{equation*}
    R^h_{(y_1,y_2,y_3)}(x_1,x_2,x_3) = (x_1 + y_1 + x_2y_3, x_2 + y_3, x_3 + y_3). 
\end{equation*}
whose derivative matrix is given by 
\begin{equation*}
    d(R^h_{(y_1,y_2,y_3)})_{(x_1,x_2,x_3)} = \begin{bmatrix}
        1 & y_3 & 0\\
        0 & 1 & 0\\
        0 & 0 & 1
    \end{bmatrix}. 
\end{equation*}

Also, it is easy to check that 
\begin{equation*}
    F(u)^{-1} = \left(\frac{u}{2} + \frac{5u^2}{6}, -u,-\frac{u}{2}\right). 
\end{equation*}
for every $u \in U$. Then, we get 
\begin{equation}\label{dF_uH}
    \hat{d}F_u(X) = \begin{bmatrix}
        1 & -\frac{u}{2} & 0 \\
        0 & 1 & 0\\
        0 & 0 & 1
    \end{bmatrix} \cdot \begin{bmatrix}
        -\frac{1}{2} - \frac{2}{3}u\\
        1 \\
        \frac{1}{2}
    \end{bmatrix} X
    = X\begin{bmatrix}
        -\frac{1}{2}-\frac{7u}{6}\\
        1 \\
        \frac{1}{2}
    \end{bmatrix}
\end{equation}

Therefore, the spanned subspace $W$ is given by  
\begin{equation*}
    W = Span\left\{X\begin{bmatrix}
        -\frac{1}{2}-\frac{7u}{6}\\
        1 \\
        \frac{1}{2}
    \end{bmatrix} \in \R^3: u \in U, X \in \R\right\}. 
\end{equation*}
which one can prove that $\dim W = 2$. Now,
considering the expression for $f_0$ in (\ref{f_0H}), we get 
\begin{equation}\label{df_0H}
    d(f_0)_{(0,0,0)}^n = \begin{bmatrix}
        1 & n & 0\\
        0 & 1 & 0\\
        0 & n & 1
    \end{bmatrix}, \forall n \in \N. 
\end{equation}
and also, taking $U = [-1,1]$ and $B = \begin{bmatrix}
    -\frac{1}{2} & 0\\
    1 & 1\\
    \frac{1}{2} & \frac{1}{2}
\end{bmatrix}$, the Kalman matrix $\mathcal{K} = (df_0^2 B \hbox{ }df_0 B \hbox{ }B)$ is given by 
\begin{equation*}
    \mathcal{K} = \begin{bmatrix}
        \frac{3}{2} & 2 & \frac{1}{2} & 1 & -\frac{1}{2} & 0\\
        1 &1 &1 &1 &1 &1\\
        \frac{5}{2} & \frac{5}{2} & \frac{3}{2} & \frac{3}{2} & \frac{1}{2} & \frac{1}{2}
   \end{bmatrix}
\end{equation*}

It is not hard to show that $rank(\mathcal{K}) = 3$, which implies that $\mathfrak{h} = \mathfrak{g}$. Therefore, the system is accessible. 

Let us check the conditions for Theorem (\ref{adrank}). In fact, by the expression (\ref{dF_uH}) we have that $\hat{d}F_0 = [-\frac{1}{2} \hbox{ }1 \hbox{ }\frac{1}{2}]^T$. Also, using the expression (\ref{df_0H}), we get  
\begin{equation*}
    \mathcal{V} = 
    \begin{bmatrix}
        \frac{3}{2} & \frac{1}{2}& -\frac{1}{2}\\
        1 & 1 & 1\\
        \frac{5}{2} & \frac{3}{2} & \frac{1}{2}
    \end{bmatrix}
\end{equation*}
which implies that $rank(\mathcal{V}) = 3$ as well. Therefore, the system $(\Sigma)$ is locally controllable. 

In particular, the fact of $\mathcal{R}$ is open and every eigenvalue of $d(f_0)_{(0,0,0)}$ are equal $1$, it follows by \cite[Theorem 69]{TMC} that $G \subset \mathcal{R} \cap \mathcal{C}$, which implies that the system (\ref{heisenberg}) is controllable. This reinforces the hypothesis that $e \in \hbox{int}\mathcal{R}$ is absolutely important to our environment.  
\end{example}

\section{Conclusion}

This paper presents an algebraic condition for the accessibility of discrete-time linear systems on Lie groups using a novel notion of the derivative of the function $u \mapsto f(u,e)$. Additionally, we establish the \textit{ad-rank condition} for these systems, inspired by the continuous case. Under this condition, we demonstrate the existence of a control set $D$ with a non-empty interior such that $D = \mathcal{C} \cap \overline{\mathcal{R}}$. We illustrate our findings with examples of systems on the affine two-dimensional Lie group $\hbox{Aff}_2(\R)$, the special matrix Lie group $\hbox{SL}(2,\R)$ and the three-dimensional Heisenberg group, all of which meet the stated accessibility and ad-rank conditions. This work builds on prior studies of continuous-time systems and extends crucial concepts to the discrete-time domain, thus bridging an essential gap in the literature on Lie group systems.

The evidence suggests that this paper makes a substantial contribution to the field of discrete-time linear systems on Lie groups by proving key accessibility and ad-rank properties. Future research could explore the robustness of these properties under perturbations or extend the analysis to more complex Lie group structures. By establishing these accessibility and ad-rank properties, our work provides a foundational understanding that can be applied to various engineering and computational problems where such systems are prevalent


\begin{thebibliography}{99} 

		\bibitem{AyalaAndTirao} AYALA, V.; TIRAO, J. \emph{Linear control systems on Lie groups and controllability.} Proceedings of symposia in pure mathematics. Vol. 64. American Mathematical Society, (1999).
		
		\bibitem{ayalaeadriano} AYALA, V.; DA SILVA, A. \emph{Controllability of linear systems on Lie groups with finite semisimple center.} SIAM Journal on Control and Optimization, Vol. 55, No. 2, pp. 1332–1343. (2017).
		
		
		
		\bibitem{AyalaandDaSilva2}AYALA, V.; DA SILVA, A. \emph{ On the characterization of the controllability property for linear control systems on nonnilpotent, solvable three-dimensional Lie groups. } Journal of Differential Equations. Vol. 266, No. 12 pp. 8233-8257, (2019).

        \bibitem{AyalaandDaSilva3}AYALA, V.; DA SILVA, A. Zsigmond, G. \emph{ Control sets of linear systems on Lie groups. } Nonlinear Differ. Equ. Appl. pp. 1-15. (2017)

        \bibitem{TAJ1}AYALA, V.; CAVALHEIRO, T. ; COSSICH, J. A. N.; SANTANA, A. J.  \emph{Controllability of discrete-time linear systems on solvable Lie groups. } arXiv preprint arXiv:2302.00145, (2024).

  
		\bibitem{TMC}CAVALHEIRO, T. M. (2024) \emph{ Controllability of discrete-time linear systems on Lie groups.} [unpublished Doctoral thesis. State University of Maringá]. Maringá, Brazil.
		
		\bibitem{TAJ}CAVALHEIRO, T.M., COSSICH, J.A.N., SANTANA, A.J. \emph{The Chain Control Set of Discrete-Time Linear Systems on the Affine Two-Dimensional Lie Group. } J Dyn Control Syst 30, 25 (2024). https://doi.org/10.1007/s10883-024-09700-5
		
		\bibitem{FritzKliemann}COLONIUS, F., AND KLIEMANN, W. \emph{The Dynamics of Control}. Springer, (1999).
		
		\bibitem{CCS1}COLONIUS, F.; SANTANA, A. J.; COSSICH, J. \emph{Outer invariance entropy for discrete-time linear systems on Lie groups.} ESAIM: Control, Optimisation and Calculus of Variations, Vol. 27, (2021).
		
		\bibitem{CCS2}COLONIUS, F.; SANTANA, A. J.; COSSICH, J. \emph{Controllability properties and invariance pressure for discrete-time linear systems}. Journal of Dynamics and Differential Equations, Vol. 34, (2022).
		
		\bibitem{adriano} DA SILVA, A. \emph{ Controllability of linear systems on solvable Lie groups. } SIAM Journal on Control and Optimization, Vol. 54, No. 1, pp. 372–390. (2016).
		
		\bibitem{adriano1} DA SILVA, A.; ROJAS, A. F. P.; \emph{Weak condition for the existence of control sets with a nonempty interior for linear control systems on nilpotent groups}. arXiv preprint arXiv:2311.07364, (2023).
		
		\bibitem{DoRocioSantanaAndVerdi}DO ROCIO, O. G.; SANTANA, A. J.; VERDI, M. A. \emph{Hermes-Sussmann and ad-rank Conditions for Local Controllability.} Journal of Dynamical and Control Systems, Vol. 23, pp. 535-545, (2017).
		
		\bibitem{elliot}ELLIOT, D. \emph{Bilinear Control Systems: Matrices in Action}.
		Springer-Verlag, (2009).
		
		
		\bibitem{JakubczykAndSontag}JAKUBCZYK, B.; SONTAG, E. \emph{Controllability of nonlinear discrete-time systems: a lie-algebraic approach}. SIAM Journal of Control and Optimization, Vol. 28, No 1, pp. 189–213, (1990).
		
		\bibitem{jouan}JOUAN, P. \emph{ Controllability of Linear Systems on Lie Groups.} Journal of Dynamical and Control Systems, Vol. 17, No. 4, (2011).
		
		\bibitem{KalmanHoNarendra}KALMAN, R. E.; HO, Y. C.; NARENDRA, K. \emph{ Controllability of linear dynamical systems.} Contributions to Differential Equations, Vol. 1, No. 2 pp. 189–213. (1963).
		
		\bibitem{kawski} KAWSKI, M. \emph{The complexity of deciding controllability}.  Systems  Control Letters, Vol. 15, No. 1, pp. 9-14, (1990).
		
		\bibitem{biosystems} DÍAZ-SEOANE, S.; BARREIRO BLAS, A.; VILLAVERDE, A. F. \emph{Controllability and accessibility analysis of nonlinear biosystems}. Computer Methods and Programs in Biomedicine, Vol. 242, (2023).
		
		\bibitem{onish}ONISHCHIK, A.; VINBERG, E. \emph{Lie groups and Lie algebras}. Springer, (1993).
		
		\bibitem{SanMartinandAyala}SAN-MARTIN, L. A. B., AYALA, V. \emph{ Controllability Properties of a Class of Control Systems on Lie Groups}. Lectures Notes in Control and Information Science, (2001). 
		
		\bibitem{sanmartin2}SAN-MARTIN, L. A. B. \emph{ Algebras de Lie.} Editora Unicamp, (2010).
		
		\bibitem{sanmartin1}SAN-MARTIN, L. A. B. \emph{ Lie groups.} Springer, (2016).
		
		\bibitem{son}SONTAG, E. \emph{Some complexity questions regarding controllability.} Proceedings of the 27th Conference on Decision and Control, pp. 1326-1329. Austin, Texas (1988).
		
		\bibitem{Son98}SONTAG, E. \emph{Mathematical Control Theory: Deterministic Finite Dimensional Systems}. Springer-Verlag, (1998).
		
		\bibitem{sontag harder} SONTAG, E. \emph{Controllability is harder to decide than accessibility}. SIAM Journal on Control and Optimization, Vol. 26, No. 5, pp. 1106-1118, (1988).
		
		\bibitem {sussmann1} SUSSMANN H. J. \emph{General theorem on local controllability}. SIAM Journal on Control and Optimization, Vol. 25, No. 1, pp.158–194, (1987).
		
		\bibitem {sussmann2} SUSSMANN H. J. \emph{Lie brackets and local controllability: a sufficient condition for scalar-input systems}. SIAM Journal on Control and Optimization, Vol. 21, No. 5, pp. 686–713, (1983).
		
	\end{thebibliography}
\end{document}